\documentclass[10pt,a4paper]{article}
\usepackage{latexsym,amssymb,amsmath, amsbsy,amsopn, amstext}
\usepackage{graphicx,epsfig,tikz,pgf,hyperref,cancel,color,float,ifpdf}
\usepackage{extarrows,algorithmic,algorithm}
\usepackage{amsmath, amsthm, amssymb,stfloats,algorithm}
\usepackage{subfigure}

\usepackage{fancyhdr}
\allowdisplaybreaks[4]

\newtheorem{definition}{\bfseries Definition}%[section]
\newtheorem{proposition}{\bfseries Proposition}%[section]
\newtheorem{theorem}{\bfseries Theorem}

\newtheorem{corollary}{\bfseries Corollary}%[section]
\newtheorem{remark}{\bfseries Remark}

\newcommand{\ubar}{{\bf u}}

\newcommand{\R}{\mathbb{R}}

\newcommand{\ep}{\epsilon}

\newcommand{\C}{\mathcal{C}}
\newcommand{\D}{\mathcal{D}}
\newcommand{\Kinfinity}{\mathcal{K}}

\newcommand{\Kzcbf}{K_{\mathrm{zcbf}}}
\newcommand{\Kclf}{K_{\mathrm{clf}}}
\newcommand{\Ce}{\mathcal{C}_\epsilon}
\newcommand{\xbar}{x}

\newcommand{\ubold}{{\bf u}}

\begin{document}
	%	\begin{spacing}{2.0}
	\title{Robustness of Control Barrier Functions for Safety Critical Control\footnote{{\bf Correction to Theorem 3 and some typos of the paper appeared in IFAC Conference on Analysis and Design of Hybrid Systems, Atlanta, GA, USA, page 54-61, 2015.}}}
	%\thanksref{footnoteinfo}}
	%\thanks[footnoteinfo]{This work is partially supported by the National Science Foundation Grants 1239055, 1239037 and 1239085.}
	\author{		Xiangru Xu, Paulo Tabuada, Jessy W. Grizzle, Aaron D. Ames		
%		\thanks{A. D. Ames is with the Woodruff School of Mechanical Engineering, the
%			School of Electrical and Computer Engineering,  Georgia Institute of Technology, Atlanta, Georgia,  email: aames@gatech.edu}
		\thanks{This research is supported by NSF CPS Awards 1239055, 1239037 and 1239085.}
		\thanks{X. Xu and J. W. Grizzle are with the Department of Electrical Engineering and Computer Science, University of Michigan, Ann Arbor, MI, email:{\tt\small \{xuxiangr,grizzle\}@umich.edu}.  P. Tabuada is with the  Department of Electrical Engineering, University of California at Los Angles, Los Angles, CA, email:{\tt\small tabuada@ucla.edu}.  A. D. Ames is with the Woodruff School of Mechanical Engineering, the
			School of Electrical and Computer Engineering,  Georgia Institute of Technology, Atlanta, Georgia,  email:{\tt\small aames@gatech.edu}. }
		}	
%		\thanks{J. W. Grizzle is with the Dept. of Electrical Engineering and Computer Science, University of Michigan, Ann Arbor, MI, email: grizzle@umich.edu}
%		\thanks{P. Tabuada is with the Dept. of Electrical Engineering, University of California at Los Angles, Los Angles, CA, email: tabuada@ucla.edu}

	\date{}

%\markboth{\bf Place name or initials here:\underline{\hspace*{1.5in}}}{\bf Place name or initials here:\underline{\hspace*{1.5in}}}

%	\begin{flushright}
%		{\bf Exam Number:}
%		%\bline{0.6in}
%	\end{flushright}

	\maketitle

\thispagestyle{fancy}          %更改plain状态
%\fancyhead{}                      %清除以前的命令
%\lhead{Submitted to IEEE VTC 2008 spring}           %左上角添加
%\chead{}
%\rhead{Correction to the paper appeared in IFAC Conference on Analysis and Design of Hybrid Systems, Atlanta, GA, USA, 54-61, 2015}
%\lfoot{}
%\cfoot{\thepage}   %current page number
%\rfoot{}
\renewcommand{\headrulewidth}{0pt}       %把页眉线的宽度设为零，即去掉页眉线

	\begin{abstract}                % Abstract of not more than 250 words.
		Barrier functions (also called certificates) have been an important tool for the verification of hybrid systems, and have also played important roles in optimization and multi-objective control. The extension of a barrier function to a controlled system results in a control barrier function. This can be thought of as being analogous to how Sontag extended Lyapunov functions to control Lyapunov functions in order to enable controller synthesis for stabilization tasks. A control barrier function enables controller synthesis for safety requirements specified by forward invariance of a set using a Lyapunov-like condition. This paper develops several important extensions to the notion of a control barrier function. The first involves robustness under perturbations to the vector field defining the system. Input-to-State stability conditions are given that provide for forward invariance, when disturbances are present, of a ``relaxation'' of set rendered invariant without disturbances. A control barrier function can be combined with a control Lyapunov function in a quadratic program to achieve a control objective subject to safety guarantees. The second result of the paper gives conditions for the control law obtained by solving the quadratic program to be Lipschitz continuous and {\color{black}therefore to gives rise to} well-defined solutions of the resulting closed-loop system. %Simulation results are  provided to illustrate the theoretical results.
	\end{abstract}


\section{Introduction}\label{sec:introduction}

Lyapunov functions are used to certify stability properties of a set without calculating the exact solution of a system. In a similar manner, barrier certificates (functions) are used to verify temporal properties (such as safety, avoidance, eventuality) of a set, without the difficult task of computing the system's reachable set; see \cite{Prajna2007siam}, \cite{prajna2007framework}. These same references show that when the vector fields of the system are polynomial and the sets are semi-algebraic, barrier certificates can be computed by sum-of-squares optimization. In the original formulation
of \cite{prajna2007framework}, all sublevel sets of the barrier certificate were required to be invariant because the derivative of the barrier certificate along solutions was required to be non-positive. This condition was relaxed by \cite{KongExpoBarrier13} and \cite{BarrierRevisited} so that tighter
over-approximations of the reachable set could be obtained, and such that more expressive barrier certificates could be synthesized using semi-definite programming. The key idea there was to only require that a single sublevel set be invariant, namely, the set of points where the barrier certificate was non-positive.

The natural extension of barrier functions to a system with control inputs is a control barrier function (CBF), first proposed by  \cite{wieland2007constructive}; this work used the original condition of a barrier function that imposes invariance of all sublevel sets. The unification of control Lyapunov functions (CLFs) with CBFs appeared at the same conference in \cite{Romd2014CDCunitiing} and \cite{aaroncbfcdc14}, using two contrasting formulations. The objective of \cite{Romd2014CDCunitiing} was to incorporate into a single feedback law the conditions required to simultaneously achieve asymptotic stability of an equilibrium point, while avoiding an unsafe set. The feedback law was constructed using Sontag's universal control formula (\cite{Sontag:universal}), provided that a  ``control Lyapunov barrier function'' inequality could be met. Importantly, if the stabilization and safety objectives were in conflict, then no feedback law could be proposed. In contrast, the approach of \cite{aaroncbfcdc14} was to pose a feedback design problem that \textit{mediates} the safety and stabilization requirements, in the sense that safety is always guaranteed, and progress toward the stabilization objective is assured when the two requirements ``are not in conflict''.

The essential difference between these two approaches is perhaps best understood through an example. A vehicle equipped with Adaptive Cruise Control (ACC) seeks to converge to and maintain a fixed cruising speed, as with a common cruise control system. Converging to and maintaining fixed speed is naturally expressed as asymptotic stabilization of a set. With ACC, the vehicle must in addition guarantee a safety condition, namely, when a slower moving vehicle is encountered, the controller must automatically reduce vehicle speed to maintain a guaranteed lower bound on time headway or following distance, where the distance to the leading vehicle is determined with an onboard radar. When the leading car speeds up or leaves the lane, and there is \textit{no longer a conflict between safety and desired cruising speed}, the adaptive cruise controller automatically increases vehicle speed. The time-headway safety condition is naturally expressible as a control barrier function. In the approach of \cite{aaroncbfcdc14}, a Quadratic Program (QP) mediates the two inequalities associated with the CLFs and CBFs; in particular, relaxation is used to make the stability objective a soft constraint while safety is maintained as a hard constraint. In this way, safety and stability do not need to be simultaneously satisfiable. On the other hand, the approach of \cite{Romd2014CDCunitiing} is only applicable when the two objectives can be simultaneously met.

A second, although less important, difference in the two approaches is that \cite{Romd2014CDCunitiing} used the more restrictive invariance condition of \cite{Prajna2007siam}, while \cite{aaroncbfcdc14} used the relaxed condition of \cite{KongExpoBarrier13}, appropriately interpreted for the type of barrier function often used in optimization, see \cite{boyd2004convex}, where the barrier function is unbounded on the boundary of the allowed set, instead of vanishing on the set boundary.

The present paper builds on previous work in two important directions. First, the robustness of barrier functions and control barrier functions under model perturbation is investigated. An Input-to-State (ISS) stability property of a safe set is established when  perturbations are present and the barrier function vanishes on the set boundary. The second result gives conditions that guarantee local Lipschitz continuity of the feedback law arising from the QP used to mediate safety and asymptotic convergence to a set. The analysis is based on the constraint qualification conditions along with the KKT conditions for optimality. While the result is applicable to the type of barrier function in \cite{aaroncbfcdc14}, it will be stated for barrier functions used in this paper that vanish on the set boundary.

The remainder of the paper is organized as follows. Section \ref{sec:barrier} defines zeroing barrier functions and zeroing control barrier functions, and establishes a robustness property to model perturbations. Section \ref{sec:LipCon} develops the conditions for the solution of the QP to be locally Lipschitz continuous in the problem data. The theory developed  is illustrated in Section \ref{sec:example} on adaptive cruise control. Section \ref{sec:conclusions} summarizes the conclusions.

\emph{Notation:} The set of real, positive real and non-negative real numbers are denoted by  $\R$, $\R^+$ and $\R^+_0$, respectively. The Euclidean norm is denoted by $\|\cdot\|$. The transpose of matrix $A$ is denoted by $A^\top$. The interior and boundary of a set $\mathcal{S}$ are denoted by $\mathrm{Int}(\mathcal{S})$ and $\partial\mathcal{S} $, respectively.
%
%the closure of  $\mathcal{S}$ is denoted by $\overline{\mathcal{S}}$ and defined as $\overline{\mathcal{S}}=\mathcal{S}\cup\partial\mathcal{S}$.
%With $\mathcal{R},\mathcal{S}\subseteq \mathbb{R}^n$, the Minkowsky sum of $\mathcal{R}$ and $\mathcal{S}$ is denoted by $\mathcal{R}\oplus\mathcal{S}=\{x\in\mathbb{R}^n\,\vert\, x=r+s \text{ for some }r\in\mathcal{R}\text{ and }s\in\mathcal{S}\}$.
%
The distance from $x$ to a set $\mathcal{S}$ is denoted by {$\Vert x\Vert_\mathcal{S}=\inf_{s\in \mathcal{S}}\Vert x-s\Vert$}.
For any essentially bounded function $g:\mathbb{R}\to \mathbb{R}^n$,  the infinity norm of $g$ is denoted by $\Vert g\Vert_{\infty}=\textrm{ess}\sup_{t\in \mathbb{R}}\Vert g(t)\Vert$.

%The closed ball in $\R^n$ with radius $\varepsilon\in \mathbb{R}^+$ and center at $0$ is denoted by $B_\varepsilon=\{x\in \mathbb{R}^n\,\vert\, \Vert x\Vert\le \varepsilon\}$.

A function $f:\R^n\rightarrow\R^m$ is called \emph{Lipschitz continuous} on $I\subset \R^n$ if there exists a constant $L\in\R^+$ such that $\|f(x_2)-f(x_1)\|\leq L\|x_2-x_1\|$ for all $x_1,x_2\in I$, and called \emph{locally Lipschitz continuous} at a point $x\in\R^n$ if there exist constants $\delta\in\R^+$ and $M\in\R^+$ such that {$\|f(x)-f(x')\|\leq M\|x-x'\|$} holds for all $\|x-x'\|\leq \delta$. A continuous function {$\beta_1:[0,a)\rightarrow [0,\infty)$} for some $a>0$ is said to belong to \emph{class $\mathcal{K}$} if it is strictly increasing and $\beta_1(0)=0$.
A continuous function {$\beta_2:[0,b)\times [0,\infty)\rightarrow [0,\infty)$} for some $b>0$ is said to belong to \emph{class $\mathcal{KL}$}, if for each fixed $s$, the mapping $\beta_2(r,s)$ belongs to class $\mathcal{K}$ with respect to $r$ and for each fixed $r$, the mapping $\beta_2(r,s)$ is decreasing with  respect to $s$ and $\beta_2(r,s)\rightarrow 0$ as $s\rightarrow \infty$.

\section{Zeroing (Control) Barrier Functions}\label{sec:barrier}
%\input{sections/barrierfunction.tex}
%\input{sections/barrierfunction.Paulo.Xu.JWG.tex}

%\textcolor{blue}{The following material should be placed in the notation section\\
%*****************\\
%The closure of a set $\mathcal{S}$ is denoted by $\overline{\mathcal{S}}$ and defined as $\overline{\mathcal{S}}=\mathcal{S}\cup\partial\mathcal{S}$. When $\mathcal{R},\mathcal{S}\subseteq \mathbb{R}^n$ we define the Minkowsky sum $\mathcal{R}\oplus\mathcal{S}$ to be the set $\{x\in\mathbb{R}^n\,\vert\, x=r+s \text{ for some }r\in\mathcal{R}\text{ and }s\in\mathcal{S}\}$. With $\Vert x\Vert_\mathcal{S}$ we denote the distance from $x$ to $\mathcal{S}$ defined by $\inf_{s\in \mathcal{S}}\Vert x-s\Vert$. For any essentially bounded function $g:\mathbb{R}\to \mathbb{R}^n$ we denote by $\Vert g\Vert_{\infty}$ the infinity norm of $g$ defined by $\textrm{ess}\sup_{t\in \mathbb{R}}\Vert g(t)\Vert$. Finally, the closed ball of radius $\varepsilon\in \mathbb{R}^+_0$ centered at $0$ is given by $B_\varepsilon=\{x\in \mathbb{R}\,\vert\, \Vert x\Vert\le \varepsilon\}$.

%\textcolor{blue}{Xiangru: can you also add the definition of $\mathcal{KL}$ function?\\
%**********************}

The  barrier function and control barrier function considered in this paper are based on \cite{KongExpoBarrier13}, \cite{BarrierRevisited}, and \cite{wieland2007constructive}. As in  \cite{aaroncbfcdc14}, the primary focus is to establish forward invariance of a given set $\C$, which one may interpret as an under approximation of the ``initial set'' and the ``safe set''  in previous formulations of barrier functions. The main contribution of the section is a robustness property under model perturbations.

%\subsection{Preliminaries}

Consider a nonlinear system on $ \R^n$,
\begin{eqnarray}
\label{eqn:dynamicalsystem}
\dot{x} = f(x),
\end{eqnarray}
with $f$ locally Lipschitz continuous. Denote by $x(t,x_0)$ the solution of~\eqref{eqn:dynamicalsystem} with initial condition $x_0\in\mathbb{R}^n$. To simplify notation, the solution is also denoted by $x(t)$ whenever the initial condition does not play an important role in the discussion. The \emph{maximal interval of existence} of $x(t,x_0)$ is denoted by $I(x_0)$. When $I(x_0)=\mathbb{R}_0^+$ for any $x_0\in \mathbb{R}^n$, the differential equation~\eqref{eqn:dynamicalsystem} is said to be \emph{forward complete}. A set $\mathcal{S}$ is called {\it forward invariant} if for every $x_0 \in \mathcal{S}$, $x(t,x_0) \in \mathcal{S}$ for all $t \in I(x_0)$.

For $\epsilon\geq 0$, define the family of closed sets $\Ce$ as
\begin{eqnarray}
\label{eqn:superlevelsetC}
\Ce &=& \{ x \in \R^n : h(x) \geq -\ep\}, \\
\label{eqn:superlevelsetC2}
\partial \Ce &=& \{ x \in \R^n : h(x) = -\ep\}, \\
\label{eqn:superlevelsetC3}
\mathrm{Int}(\Ce) &=& \{ x \in \R^n : h(x) > -\ep\},
\end{eqnarray}
where {$h: \R^n \to \R$} is a continuously differentiable function. By construction, $\C_{\ep_1}\subset\C_{\ep_2}$ for any $\ep_2>\ep_1\geq 0$. For simplicity, the set $\C_0$ is denoted by $\C$.
%Later, it will be assumed that
%$\mathrm{Int}({\C}) \not=\emptyset$ and $\overline{\mathrm{Int}(\mathcal{C})} = \mathcal{C}$.

The definition of a barrier function is made easier through an appropriate extension of the notion of
class $\mathcal{K}$ function.

\begin{definition}\label{def:extend} (Based on \cite{KHALIL01})  A continuous function {$\beta:(-b,a)\rightarrow (-\infty,\infty)$} for some $a,b>0$ is said to belong to \emph{extended class $\mathcal{K}$} if it is strictly increasing and $\beta(0)=0$.
	%A continuous function $\beta_2:(-b,a)\times [0,\infty)\rightarrow (-\infty,\infty)$ for some $a,b>0$ is said to belong to {\emph{extended class $\mathcal{KL}$}}, if for each fixed $s$, the mapping $\beta_2(r,s)$ belongs to extended class $\mathcal{K}$ with respect to $r$ and for each fixed $r$, the mapping $|\beta_2(r,s)|$ is decreasing with  respect to $s$ and $|\beta_2(r,s)|\rightarrow 0$ as $s\rightarrow\infty$.
\end{definition}

%The following lemma is a straightforward generalization of Lemma 4.4 of \cite{KHALIL01}.  The proof is omitted.
%
%\begin{lem}\label{lem:compare}
%Consider the following autonomous scalar differential equation
%$$
%\dot{y}=-\alpha_0(y),\;y(t_0)=y_0
%$$
%where $\alpha_0$ is a locally Lipschitz extended class $\mathcal{K}$ function defined on $(-b,a)$ where $a,b>0$. Then, for all \mbox{$y_0\in(-b,a)$}, this equation has a unique solution $y(t)$ defined for all $t\geq t_0$ and moreover, $y(t)=\sigma(y_0,t-t_0)$ where $\sigma$ is a extended class $\mathcal{KL}$ function.
%\end{lem}

\subsection{Zeroing Barrier Functions}

The class of barrier functions considered in this paper is defined as follows.

\begin{definition}
	\label{def:barrierfunctions2}
	Consider a  dynamical system \eqref{eqn:dynamicalsystem} and the set $\C$ defined by \eqref{eqn:superlevelsetC}-\eqref{eqn:superlevelsetC3} for some continuously differentiable function $h: \R^n\rightarrow \R$. If there exist a locally Lipschitz extended class $\Kinfinity$ function $\alpha$ and a set $\D$ with $\C\subseteq\mathcal{D}\subset \R^n$ such that
	\begin{align}
	L_f h(x)& \geq  -\alpha(h(x)),\forall \; x \in \D,\label{eqn:generalinequality3}
	\end{align}
	then the function $h$ is called a \emph{zeroing barrier function (ZBF)}.
\end{definition}

Existence of a ZBF implies the forward invariance of $\mathcal{C}$, as shown by the following theorem.
%Similar to Theorem 1 in \cite{aaroncbfcdc14}, we have the following result.

\begin{theorem}\label{theorem:GBF}
	Given a dynamical system \eqref{eqn:dynamicalsystem} and a set $\C$ defined by \eqref{eqn:superlevelsetC}-\eqref{eqn:superlevelsetC3} for some continuously differentiable function $h: \R^n\rightarrow \R$, if $h$ is a ZBF defined on the set $\D$ with $\C\subseteq\mathcal{D}\subset \R^n$, then $\C$ is forward invariant.
\end{theorem}
%\textbf{[Let's give the easy proof? Include invariance of $\C_\ep \subset D$ ?]}
\begin{proof}
	Note that for any $x \in \partial\C$, $L_f h(x)\geq -\alpha(h(x))=0$. According to Nagumo's theorem (\cite{BlanchiniBook08}), the set $\C$ is forward invariant.\hfill$\Box$
	%Similarly, for any $\ep\geq 0$ such that $\C_\ep\subseteq\mathcal{D}$,  for any $x \in \partial\Ce$, we have $L_f h(x)\geq -\alpha(h(x))=-\alpha(-\ep)>0$. Again, Nagumo's theorem implies that the set $\Ce$ is forward invariant.
\end{proof}

Recall that  the original barrier condition in \cite{prajna2007framework} requires that $\dot{h}\geq 0$, when expressed in the notation of the present paper, which implies that all superlevel sets of $h$ inside $\C$
are  invariant. As in \cite{BarrierRevisited}, \cite{KongExpoBarrier13} and \cite{aaroncbfcdc14}, inequality \eqref{eqn:generalinequality3}  relaxes the conventional condition by requiring a single superlevel set of $h$, which is $\C$ itself, to be invariant.

%Furthermore, any superlevel set of $h$ that contains $\C$ and contained in $\D$ is {\color{cyan}contractive} (\cite{BlanchiniBook08}), as shown in the following corollary.
%
%\begin{cor}
%Under the condition of Theorem \ref{theorem:GBF}, if $h$ is a ZBF defined on the set $\D$ with $\C\subseteq\mathcal{D}\subset \R^n$, then, for any $\ep> 0$ such that $\C_\ep\subseteq\mathcal{D}$, $\C_\ep$ is {\color{cyan}contractive.}
%\end{cor}

\subsection{Robustness Properties of ZBFs}

In this section, the extent to which forward invariance of the set $\C$, asserted in Theorem~\ref{theorem:GBF}, is robust with respect to different perturbations on the dynamics~\eqref{eqn:dynamicalsystem} is investigated. This will be accomplished by showing that existence of a ZBF implies asymptotic stability of the set $\C$.

Recall that a closed and forward invariant set $\mathcal{S}\subseteq \mathbb{R}^n$ is said to be locally asymptotically stable for a forward complete system~\eqref{eqn:dynamicalsystem} if there exist an open set $\mathcal{R}$ containing $\mathcal{S}$ and a class $\mathcal{KL}$ function $\beta$ such that for any $x_0\in \mathcal{R}$
\begin{equation}
\label{Eq:SetStability}
\Vert x(t,x_0)\Vert_{\mathcal{S}}\le \beta\left(\Vert x_0\Vert_{\mathcal{S}},t\right).
\end{equation}
Whenever the set $\mathcal{S}$ is compact, inequality~\eqref{Eq:SetStability} implies $I(x_0)=\mathbb{R}_0^+$ for all $x_0\in \mathcal{R}$. Therefore, the forward completeness assumption on~\eqref{eqn:dynamicalsystem} is no longer needed. Note that asymptotic stability of $\mathcal{S}$ implies invariance of $\mathcal{S}$ as can be seen by noting that $x_0\in\mathcal{S}$ implies $\Vert {\color{black}x_0}\Vert_\mathcal{S}=0$ and $\beta(\Vert {\color{black}x_0}\Vert_\mathcal{S},t)=0$ which, in turn, implies $\Vert x(t,x_0)\Vert_\mathcal{S}{\color{black}=} 0$ and $x(t,x_0)\in \mathcal{S}$.

Once asymptotic stability of $\C$ is established, several robustness results in the literature will be used to characterize the robustness of forward invariance of the set $\C$. The critical observation, upon which all the results in this section rely, is that, if $\mathcal{D}$ is open, then a ZBF $h$ induces a Lyapunov function $V_\C:\mathcal{D}\to \mathbb{R}_0^+$ defined by:
\begin{equation}
\label{Eq:V}
V_\C(x)=\left\{\begin{array}{ccc}
0, & \text{if} & x\in \C,\\
-h(x), & \text{if} & x\in\mathcal{D}\backslash\C.\end{array}\right.
\end{equation}
It is easy to see that: \textbf{1)} $V_\C(x)=0$ for $x\in \C$; \textbf{2)} $V_\C(x)>0$ for $x\in\mathcal{D}\backslash\C$; and \textbf{3)} $L_f V_\C(x)$ satisfies the following inequality for $x\in\mathcal{D}\backslash \C$:
$$L_f V_\C(x) =  -L_f h(x)\le\alpha\circ h(x) = \alpha(-V_\C(x))<0,$$
where $\alpha$ is the locally Lipschitz extended class $\Kinfinity$ function introduced in Definition \ref{def:barrierfunctions2}.
%\textbf{[Xiangru: Is $V$ a ``good'' Lyapunov function?  the class $K$ function of $\|x\|$ that upper and lower bound $V$ can not be easily found (for computation purposes)?]}
It thus follows from these three properties, from the fact that $V_\C$ is continuous on its domain and continuously differentiable at every point $x\in\mathcal{D}\backslash \C$, and from\footnote{While Theorem 2.8 requires the function $V$ to be smooth, $V$ can always be smoothed as shown in Proposition 4.2 in~\cite{SmoothConverse96}.} Theorem 2.8 in~\cite{SmoothConverse96} that the set $\C$ is asymptotically stable whenever~\eqref{eqn:dynamicalsystem} is forward complete or the set $\C$ is compact. The preceding discussion is summarized in the following result.

\begin{proposition}
	\label{Prop:SetStability}
	Let $h:\mathcal{D}\to\mathbb{R}$ be a continuously differentiable function defined on an open set $\mathcal{D}\subseteq\mathbb{R}^n$. If $h$ is a ZBF for the dynamical system~\eqref{eqn:dynamicalsystem}, then the set $\C$ defined by $h$ is asymptotically stable. Moreover, the function $V_\C$ defined in \eqref{Eq:V} is a Lyapunov function.
	%Lyapunov function~\eqref{Eq:V} implies asymptotic stability of the set  $\C$  whenever~\eqref{eqn:dynamicalsystem} is forward complete or the set $\C$ is compact.
\end{proposition}

The relationships between asymptotic stability and different robustness properties are well documented in the literature. For the reader's benefit, the following proposition paraphrases several existing results using the notation of this paper.
%we paraphrase several existing results using the notation of this paper.

\begin{proposition}
	\label{Prop:Robust}
	Under the assumptions of Proposition~\ref{Prop:SetStability} the following statements hold:
	\begin{itemize}
		\item There exist $\varepsilon\in \mathbb{R}^+_0$ and class $\mathcal{K}$ function {$\sigma:[0,\varepsilon]\to \mathbb{R}_0^+$} such that for any continuous function $g_1:\mathbb{R}^n\to\mathbb{R}^n$ satisfying $\Vert g_1(x)\Vert\le \sigma\left(\Vert x\Vert_{\C}\right)$ for $x\in \mathcal{D}\backslash\mathrm{Int}(\mathcal{C})$, the set $\C$ is still asymptotically stable for the system $\dot{x}=f(x)+g_1(x)$ describing the effect of a disturbance modeled by $g_1$ on system \eqref{eqn:dynamicalsystem}.
		\item There exist a constant $k\in\mathbb{R}^+$ and class $\mathcal{K}$ function $\gamma$ such that the set $\mathcal{C}_{\gamma\left(\Vert g_2\Vert_\infty\right)}\subseteq \mathcal{D}$ is locally asymptotically stable for the system $\dot{x}=f(x)+g_2(t)$ describing the effect of a  disturbance modeled by $g_2$, and satisfying $\Vert g_2\Vert_\infty\le k$, on system~\eqref{eqn:dynamicalsystem}.
		%solution $x(t,x_0)$ of the differential equation $\dot{x}=f(x)+{\color{red}g_2(t)}$ describing the effect of a  disturbance modeled by $g_2$ on the system~\eqref{eqn:dynamicalsystem} satisfies:
		%\begin{equation}
		%\label{Eq:ISS}
		%{\color{red}\Vert x(t,x_0)\Vert_{\C}\le \gamma\left(\Vert g_2\Vert_\infty\right)},
		%\end{equation}
		%for any {\color{cyan}$x_0\in \mathcal{C}_\varepsilon$ for which $\mathcal{C}_\varepsilon\oplus B_{\gamma\left(\Vert g_2\Vert_\infty\right)}\subseteq \mathcal{D}$.?}
	\end{itemize}
\end{proposition}

The first result in Proposition~\ref{Prop:Robust} corresponds to Theorem~2.8 in~\cite{LyapunovFunctions}. A disturbance satisfying the inequality $\Vert g_1(x)\Vert\le \sigma\left(\Vert x\Vert_{\C}\right)$ is called a vanishing perturbation since its magnitude decreases as the state $x$ approaches the set $\mathcal{C}$ and it vanishes on the boundary of $\mathcal{C}$. For this type of perturbation, the set $\C$ remains invariant. Moreover, even if a disturbance pushes the state into $\mathcal{D}\backslash\C$, the set $\C$ is asymptotically reached.

The second result in Proposition~\ref{Prop:Robust} corresponds to the observation that the system $\dot{x}=f(x)+u$ is locally input-to-state stable when $u$ is seen as a disturbance input. In this case, the disturbance $u(t)=g_2(t)$ is called a non-vanishing perturbation and the only assumption is that it is sufficiently small, in the sense that $\Vert g_2\Vert_\infty\le k$. %Inequality~\eqref{Eq:ISS} ensures invariance of the set $\C\oplus \mathcal{B}_{\gamma\left(\Vert g_2\Vert_\infty\right)}$ that can be seen as an inflation of $\C$.
Note that the ``size'' of the new asymptotically stable set $\mathcal{C}_{\gamma\left(\Vert g_2\Vert_\infty\right)}$, as measured by $\gamma\left(\Vert g_2\Vert_\infty\right)$, is an increasing function of the disturbance bound $\Vert g_2\Vert_\infty$.
%this inflation depends on the bound $\Vert g_2\Vert_\infty$ on the disturbance $g_2$.
Similarly to vanishing perturbations, if a disturbance pushes the state into $\mathcal{D}\backslash\mathcal{C}_{\gamma\left(\Vert g_2\Vert_\infty\right)}$, the set $\mathcal{C}_{\gamma\left(\Vert g_2\Vert_\infty\right)}$ is asymptotically reached.

\subsection{Zeroing Control Barrier Functions}\label{subsec:ZCBF}
Consider an affine control system of the form
\begin{eqnarray}
\label{eqn:controlsys}
\dot{x} = f(x) + g(x) u,
\end{eqnarray}
with $f$ and $g$ locally Lipschitz continuous, $x \in \R^n$ and $u \in U \subset \R^m$.

\begin{definition}\label{dfn:newcbf}
	Given a set $\mathcal{C} \subset \R^n$ defined by \eqref{eqn:superlevelsetC}-\eqref{eqn:superlevelsetC3} for a continuously differentiable function $h: \R^n \to \R$,  {\color{black}the function $h$ is called} a \emph{zeroing control barrier function (ZCBF)} defined on   set $\mathcal{D}$ with $\C\subseteq\mathcal{D}\subset \R^n$, if  there exists an extended class $\Kinfinity$ function $\alpha$ such that
	\begin{align}\label{ineq:ZCBF}
	& \sup_{u \in U}  \left[ L_f h(x) + L_g h(x) u + \alpha(h(x))\right] \geq 0,\;\forall x \in \mathcal{D}.
	\end{align}
	The ZCBF $h$ is said to be locally Lipschitz continuous if $\alpha$ and the derivative of $h$ are both locally Lipschitz continuous.
\end{definition}

If $U=\R^m$ and $L_gh(x) \ne 0$ for $x\in\mathcal{D}$, then the function $h$ is always a ZCBF.

Given a ZCBF $h$, define the set for all $x\in\D$
% of all control values that satisfy \eqref{eq:Binf}:
\begin{equation}\label{zcbfinputset}
\Kzcbf(x) =  \{ u \in U : L_f h(x) + L_g h(x) u + \alpha(h(x)) \geq 0\}. \nonumber
\end{equation}

Similar to Corollary 1 in \cite{aaroncbfcdc14}, the following result that guarantees the forward invariance of $\mathcal{C}$ can be given.

\begin{corollary}\label{cor:zbf}
	Given a set $\mathcal{C} \subset \R^n$ defined by \eqref{eqn:superlevelsetC}-\eqref{eqn:superlevelsetC3} for a continuously differentiable function $h$, if $h$ is a ZCBF on $\D$, then any Lipschitz continuous controller $u: \mathcal{D} \to U$ such that $u(x) \in \Kzcbf(x)$ will render the set $\mathcal{C}$ forward invariant.
\end{corollary}

%\begin{remark}
%We simply let $\D=\R^n$ in implementations...
%\end{remark}

Inspired by the pointwise minimum-norm controller in \cite{FreemanSIAM96} for rendering a control Lyapunov function negative definite, consider a control input of minimum norm that meets the control barrier function inequality in \eqref{ineq:ZCBF}. When the norm arises from an inner product, the resulting controller is the solution of a quadratic program (QP). The QP perspective is especially interesting because it allows the unification of performance and safety (\cite{aaroncbfcdc14}). Specifically, the inequality for a control Lyapunov function (CLF) can be added as an additional soft constraint via a relaxation parameter, while the control barrier function inequality is maintained as a hard constraint for guaranteed safety. The question arises, however, is such a feedback law locally Lipschitz continuous?  Conditions that ensures local Lipschitz continuity will be discussed in the next section.

\section{Lipschitz Continuity of a Quadratic Program for Safety and Performance}\label{sec:LipCon}
%\input{sections/LipConQP_JWG.tex}
%A control barrier function enables controller synthesis
%for the forward invariance of a set using a Lyapunov-like condition (such as \eqref{ineq:ZCBF}); particularly, it  can be unified with a control Lyapunov function in a quadratic program (QP) framework, which allow us to achieve the control objective and the safety guarantee at the same time (\cite{AaronCLFHyZerTAC14,HsuBackstepping15}).

%Inspired by the minimum-norm controller in \cite{FreemanSIAM96} for rendering a control Lyapunov function negative definite, this section considers a control input of minimum norm that meets the control barrier function inequality in \eqref{ineq:ZCBF}. When the norm arises from an inner product, the resulting controller is the solution of a quadratic program (QP). The QP perspective is especially interesting because it allows the unification of performance and safety (\cite{aaroncbfcdc14}). Specifically, the inequality for an exponentially stabilizing control Lyapunov function (ES-CLF) can be added as an additional soft constraint via a relaxation parameter, while the control barrier function inequality is maintained as a hard constraint for guaranteed safety.

The main result of this section provides sufficient conditions for  a QP-based feedback controller to be locally Lipschitz continuous, as required in Corollary 1 of \cite{aaroncbfcdc14} and Corollary \ref{cor:zbf} in Subsection \ref{subsec:ZCBF}. It will be assumed throughout this section that $U=\R^m$.

\subsection{Quadratic Program Only With the Control Barrier Constraint}

For an affine control system \eqref{eqn:controlsys} and a set $\C \subset \R^n$ defined by \eqref{eqn:superlevelsetC}-\eqref{eqn:superlevelsetC3}, consider the set of controllers $u(x)\in\Kzcbf(x)$ meeting the control barrier function condition in \eqref{ineq:ZCBF}. The controller that pointwise minimizes the Euclidean norm can be found by solving the following parameterized quadratic program % $\mathcal{P}_1(x)$
\begin{align}
\mathcal{P}_1(x):&\quad \forall~x\in\D,~ \nonumber \\
&\quad u^*(x) =  \underset{u\in\R^{m}}{\operatorname{argmin}}   \;
u^\top u, \nonumber\\
&\quad \mathrm{s.t.}  \quad L_g h(x) u+L_f h(x) +\alpha(h(x))  \geq 0, \label{QP:ZCBF2}
\end{align}
where $u\in\R^m$ is the control input and constraint \eqref{QP:ZCBF2} is the ZCBF condition shown in \eqref{ineq:ZCBF}.

%\textbf{Define explicitly what ``continuously differentiability'' and ``Lipschitz continuity'' mean for parameterized QP}

The following result establishes  the key condition for $u^\ast(x)$ to be locally Lipschitz continuous: the control barrier function should be relative degree one uniformly on $\D$ in the sense that $L_g h$ does not vanish on $\D$.

\begin{theorem}\label{prop:QPLip}
	Assume that vector fields $f$ and $g$ in the control system \eqref{eqn:controlsys} are both locally Lipschitz continuous, and that $h:\D \to \R$ is a locally Lipschitz continuous ZCBF. Suppose furthermore that the relative degree one condition, $L_g h(x) \neq 0$ for all $x \in \D$, holds. Then the solution, $u^*(x)$, of $\mathcal{P}_1(x)$ is locally Lipschitz continuous for $x \in \D$.
\end{theorem}
\begin{proof}
	%The QP $\mathcal{P}_1(x)$ is always feasible because {$U=\R^m$}; furthermore,
	Because $L_gh(x)\neq 0$ for $x\in \D$, the linear independent constraint qualification condition is satisfied (\cite{bertsekas1999nonlinear}). Hence, the KKT optimality conditions imply there exists $\mu(x)\geq 0$ such that $u^*(x)$ and $\mu(x)$ satisfy
	\begin{equation*}
	\begin{cases}
	{u^*(x)}^\top =\mu(x) L_gh(x), \\
	L_f h(x) + L_g h(x)u^*(x) +\alpha(h(x))\geq 0,\\
	\mu(x)=0\;\mbox{if}\; L_f h(x) + L_g h(x)u^*(x) +\alpha(h(x))>0.
	\end{cases}
	\end{equation*}
	Because the objective is convex and the inequality constraints are affine, the KKT necessary conditions are also sufficient (pg.~244 in \cite{boyd2004convex}). Hence, the closed form expression for $u^*(x)$ can be derived as
	%\begin{align*}
	%u^*(x)&=\mu(x)L_gh(x)^\top ,
	%\end{align*}
	%where
	\begin{equation*}
	u^*(x)=\left\{\begin{array}{l}
	0,\quad\mbox{if}\;L_f h(x)+\alpha(h(x))>0,\\
	-\frac{(L_f h(x) + \alpha(h(x)))L_gh(x)^\top }{L_gh(x)L_gh(x)^\top },\quad\mbox{otherwise.}
	%L_f h+\alpha(h)\leq 0.
	\end{array}\right.
	\end{equation*}
	%\begin{equation*}
	%u^*(x)=\left\{\begin{array}{lr}
	%-H_1^{-1}F_1,& \mbox{if}\;\omega(x)>0,\\
	%\mu(x)H_1^{-1}L_gh(x)^\top -H_1^{-1}F_1,&\mbox{if}\; \omega(x)\leq 0,
	%\end{array}\right.
	%\end{equation*}
	%with
	%\begin{align}
	%\omega(x)&=L_f h(x) + h(x).
	%%\mu(x)&=-\frac{\omega(x)}{L_gh(x)H_1^{-1}L_gh(x)^\top }
	%\end{align}
	
	The following facts about Lipschitz continuous functions are recalled.
	
	\emph{Fact 1.}
	If $f_1$ and $f_2$ are locally Lipschitz continuous on a set $I$, then whenever their sum, $f_1+f_2$, or product, $f_1f_2$, makes sense,  they are each locally Lipschitz continuous on $I$. Furthermore, if $f_3$ is real valued, then in a neighborhood of any point $x\in I$ where $f_3(x) \neq0$, the reciprocal $1/f_3$ is locally Lipschitz.
	
	\emph{Fact 2.}
	If $f_1$ is locally Lipschitz continuous on a set $I_1$ and $f_2$ is locally Lipschitz continuous on a set $I_2$ such that $f_1(I_1)\subset I_2$, then the composition $f_2 \circ f_1$ is locally Lipschitz continuous on $I_1$.
	
	With these facts in mind, define
	\begin{align}
	\omega_1(r)&=\left\{\begin{array}{rl}
	0,& \mbox{if}\;r>0, ~\\
	r,&\mbox{if}\; r\leq 0,
	\end{array} r\in \mathbb{R},\right.\label{LC:omega1}\\
	\omega_2(x)&=L_f h(x) + \alpha(h(x)),\;x\in\D,\nonumber\\
	\omega_3(x)&=-\frac{L_gh(x)^\top }{L_gh(x)L_gh(x)^\top },\; x\in\D.\nonumber
	\end{align}
	
	The function $\omega_1(r)$ is clearly Lipschitz continuous. Because $f$ and $g$ are locally Lipschitz continuous and the derivative of $h$ is  locally Lipschitz continuous, both $L_fh$ and $L_gh$ are locally Lipschitz continuous on $\D$ by \emph{Fact 1}. The same fact implies that $\omega_2$ and $L_gh L_gh^\top$ are locally Lipschitz continuous on $\D$. Furthermore, because $L_gh(x)\neq 0$ for $x\in\D$, it follows that $L_gh(x) L_gh(x)^\top \neq 0$ and thus $\omega_3(x)$ is also locally Lipschitz continuous by \emph{Fact 1}.
	
	The proof is completed by noting that
	\begin{align*}
	u^\ast(x)=\omega_1(\omega_2(x))\omega_3(x),\;x\in\D.
	\end{align*}
	Because $\omega_1(\omega_2(x))$ is locally Lipschitz continuous with respect to $x\in\D$ by \emph{Fact 2}, its product with $\omega_3(x)$ is locally Lipschitz continuous by \emph{Fact 1}, and thus $u^\ast(x)$ is locally Lipschitz continuous with respect to $x \in \D$.
	\hfill$\Box$
\end{proof}

\begin{remark}\label{rem:quadratic}
	If the objective function of $\mathcal{P}_1(x)$ is changed to $\frac{1}{2}u^\top Hu+F^\top u$, where $H$ is an $m\times m$ positive definite matrix and $F$ is an $m\times 1$ column vector, then the solution of the modified QP is also locally Lipschitz continuous with respect to $x\in\D$.
\end{remark}

\subsection{Quadratic Program Incorporating both Control Barrier and Lyapunov Constraints}
Suppose now that the desired performance of the system \eqref{eqn:controlsys} can be captured by  a CLF $V$,  as in \cite{AaronCLFHyZerTAC14,aaroncbfcdc14}.
This yields the set of control inputs that stabilize the system \eqref{eqn:controlsys}, namely
\begin{align}\label{ineq:ESCLF}
\Kclf(x)=\lbrace u\in \R^m:L_fV(x)+L_gV(x)u+cV(x)<0 \rbrace,
\end{align}
%
%where $c$ is a positive constant setting the speed of convergence.
The minimum-norm controller of Freeman and Kokotovic chooses pointwise in $x$ the element of $\Kclf(x)$ that minimizes the Euclidean norm. This is now combined with the control barrier function inequality.

In particular, given a CLF $V$ and a ZCBF $h$ with relative degree 1 in $\D$, the two ``specifications'' are  combined via the following parameterized quadratic program
\begin{align}
\mathcal{P}_2(x):&\quad \forall~x\in\D,~ \nonumber \\
&\quad \ubar^*(x)=  \underset{\ubar = \left[ u^\top,\delta \right]^\top\in\R^{m+1}}{\operatorname{argmin}}
\ubar^\top  \ubar \nonumber\\
%&\quad \ubar^*(x)=  \underset{\ubar = \left[ \begin{array}{c} u \\ \delta \end{array} \right]\in\R^{m+1}}{\operatorname{argmin}}
%\ubar^\top  \ubar \nonumber\\
&\quad \mathrm{s.t.}  \;  L_g V(x) u +L_f V(x) +cV(x)- \delta\leq 0,  \label{QP:ESCLF}\\
&\quad \quad\;\; L_g h(x) u+L_f h(x) +\alpha(h(x))   \geq 0 ,\label{QP:ZCBF}
\end{align}
where $c$ is a positive constant, $u\in\R^m$ is the control input, {\color{black}$\delta$} is a relaxation parameter\footnote{A weight is traditionally used on the relaxation parameter. This is taken care of after the proof of the main result.}, constraint \eqref{QP:ZCBF} is the ZCBF condition and constraint \eqref{QP:ESCLF} is the CLF condition.
% where
%\begin{align}
%\psi_{0}(x) &= L_f V(x) + c V(x), \\
%\psi_{1}(x) & =  L_g V(x).
%\end{align}

\begin{remark}
	The QP $\mathcal{P}_2(x)$ is always feasible, because $L_gh\neq 0$ ensures that there exists $u$ such that \eqref{QP:ZCBF} holds, which implies that the safety guarantee can always be satisfied, while the relaxation parameter $\delta$  ensures that \eqref{QP:ESCLF} can always be satisfied. Due to the relaxation parameter, the performance objective, such as asymptotic stabilization to an equilibrium point, may not necessarily be achieved.
	When the control objective and the safety guarantee are not conflicting---\textit{and a weight is appropriately added to the objective function}---the solution will result in $\delta \approx 0$. Indeed, if the objective function is $u^\top u+k^2\delta^2$ with $k\neq 0$ the weight for $\delta$, and $\hat\ubar=(\hat u^\top,0)^\top$ is a feasible point for constraints \eqref{QP:ESCLF} and \eqref{QP:ZCBF}, then the optimal solution $\ubar^*=({u^*}^\top,\delta^*)^\top$ satisfies ${u^*}^\top{u^*}+k^2{\delta^*}^2\leq {\hat u}^\top\hat u$, which implies that ${\delta^*}^2\leq {\hat u}^\top\hat u/k^2$. Therefore, $\delta^*$ can be made arbitrarily small if sufficiently large weight $k$ is chosen.
\end{remark}

%If $H$ is positive definite, it is easy to see that $H,F$ does not affect the Lipschitz continuity of the solution of $\mathcal{P}_2(x)$. Therefore, we can assume $H=2I_{m+1},F=0_{m+1}$ in $\mathcal{P}_2(x)$ for simplicity and consider the following QP:

The following theorem is the main result of this subsection.

\begin{theorem}\label{theorem:LocLip}
	Let $V$ be a CLF for the control system \eqref{eqn:controlsys} with the derivative of $V$ locally Lipschitz continuous.  Assume that the vector fields $f$ and $g$ in the control system \eqref{eqn:controlsys} are both locally Lipschitz continuous and that $h:\D \to \R$ is a locally Lipschitz continuous ZCBF. Suppose furthermore that the relative degree one condition, $L_g h(x) \neq 0$ for all $x \in \D$, holds.
	Then the solution, $\ubar^*(x)$, of $\mathcal{P}_2(x)$ is locally Lipschitz continuous for $x \in \D$.
\end{theorem}
\begin{proof}

	The proof is based on \cite{LuenOptiBook}(Chapter~3), which as a special case includes minimization of a quadratic cost function subject to affine inequality constraints.
	
	Define
	\begin{align*}
		& y_1(\xbar)=[L_g V(\xbar),-1]^\top,\;p_1(\xbar)= -L_f V(\xbar)-cV(x),\\
		& y_2(\xbar)=[L_g h(\xbar),0]^\top,\;p_2(\xbar)=-L_f h(\xbar) +\alpha(h(\xbar)),
	\end{align*}
	and note that for all $\xbar\in \mathcal{D}$,  $y_1(\xbar)$ and $y_2(\xbar)$ are linearly independent in $\mathbb{R}^{m+1}$.
	
%	Because $H(x)$ is locally Lipschitz continuous and positive definite, its inverse exists and is locally Lipschitz continuous. Define
%	$$
%	\begin{bmatrix} \bar{y}_1(\xbar), \bar{y}_2(\xbar) \end{bmatrix} = H(\xbar)^{-1} \begin{bmatrix} {y}_1(\xbar), {y}_2(\xbar) \end{bmatrix},
%	$$
%	$$
%	\begin{bmatrix} \bar{p}_1(\xbar) \\ \bar{p}_2(\xbar) \end{bmatrix} =  \begin{bmatrix} {p}_1(\xbar) \\ {p}_2(\xbar) \end{bmatrix} -  \begin{bmatrix} {y}_1(\xbar)^\top \\  {y}_2(\xbar)^\top \end{bmatrix} \bar{\ubold}(\xbar),
%	$$
%	and
%	\begin{align*}
%		%\label{eqn:Urelated2V}
%		\bar{\ubold}(\xbar)&:= -H(\xbar)^{-1}F(\xbar)\\
%		\vbold &:=  \ubold-\bar{\ubold}(\xbar).
%	\end{align*}
%	Finally, let$\left<\cdot, \cdot \right>$ define an inner product on $\mathbb{R}^{ m+1}$ with weight matrix $H(x)$ so that
%	$
%	\left< \vbold,  \vbold \right>: = \vbold^\top H(\xbar)\vbold.
%	$

	The optimization problem $\mathcal{P}_2(x)$  is then equivalent to
	\begin{align}
%		\mathcal{P}_2(x):&\quad \forall~x\in\D,~ \nonumber \\
	\ubar^*(x)&=  \underset{\ubar = \left[ u^\top,\delta \right]^\top\in\R^{m+1}}{\operatorname{argmin}}
		\ubar^\top  \ubar \label{eqn:QPCLFCBF:equivalent}\\
%		%&\quad \ubar^*(x)=  \underset{\ubar = \left[ \begin{array}{c} u \\ \delta \end{array} \right]\in\R^{m+1}}{\operatorname{argmin}}
%		%\ubar^\top  \ubar \nonumber\\
%		&\quad \mathrm{s.t.}  \;  L_g V(x) u +L_f V(x)- \delta\leq 0,  \label{QP:ESCLF}\\
%		&\quad \quad\;\; L_g h(x) u+L_f h(x) +\alpha(h(x))   \geq 0 ,\label{QP:ZCBF}
%		\ubold^*(\xbar) =   & \underset{\vbold \in \R^{m+1}}{\operatorname{argmin}}  \left<\vbold, \vbold \right> \label{eqn:QPCLFCBF:equivalent}\\
		\mathrm{s.t.} &  \left<{y}_1(\xbar),  \ubar \right>  \le {p}_1(\xbar), \nonumber \\
		& \left<{y}_2(\xbar),  \ubar \right> \le {p}_2(\xbar). \nonumber
	\end{align}
%	with
%	\begin{align}
%		\label{eqn:Urelated2V}
%		\ubold^*(\xbar)&=\vbold^*(\xbar) + \bar{\ubold}(\xbar).
%	\end{align}
	
	From  \cite{LuenOptiBook}(Chapter~3), the solution to \eqref{eqn:QPCLFCBF:equivalent} is computed as follows. Let $G(\xbar)=[G_{ij}(\xbar)]=[\langle {y}_i(\xbar), {y}_j(\xbar)\rangle]$, $i,j=1,2$ be the Gram matrix.  Due to the linear independence of $\{{y}_1(\xbar), {y}_2(\xbar)\}$,  $G(\xbar)$ is positive definite. The unique solution to \eqref{eqn:QPCLFCBF:equivalent} is
	\begin{equation}
		\label{eqn:Lambdas}
		\ubold^*(\xbar) = \lambda_1(\xbar)  {y}_1(\xbar) +  \lambda_2(\xbar)  {y}_2(\xbar) ,
	\end{equation}
	where $\lambda(\xbar)=[\lambda_1(\xbar), \lambda_2(\xbar)]^\top$ is the unique solution to
	\begin{align}
		\label{eqn:LuenbrgerFormulation}
		G(\xbar) \lambda(\xbar) & \le {p}(\xbar), \nonumber \\
		\lambda(\xbar) & \le 0,\\
		[G(\xbar) \lambda(\xbar)]_i & < {p}_i(\xbar) ~ \Rightarrow~\lambda_i(\xbar)=0, \nonumber
	\end{align}
	where $[\cdot]_i$ denotes the $i$-th row of the quantity in brackets,  $p(\xbar)=[p_1(\xbar),p_2(\xbar)]^\top$, and the inequalities hold componentwise.
	Because $G(\xbar) $ is $2 \times 2$, a closed form solution can be given. Define the Lipschitz continuous function
	$$
	\omega(r)=\left\{\begin{array}{rl}
	0,& \mbox{if}\;r>0, ~\\
	r,&\mbox{if}\; r\leq 0.
	\end{array} r\in \mathbb{R}. \right.
	$$
	For $\xbar\in\mathcal{D}$,
	%%and dropping the argument $\xbar$ for compactness of notation,
	$\lambda_1 ,  \lambda_2$ can be expressed in closed form as\newline
	\noindent \textbf{If:} $G_{21}(\xbar) \omega ({p}_2(\xbar))-G_{22}(\xbar){p}_1(\xbar)<0,$
	\begin{align}
		\label{eqn:IfCondition}
		\left[  \begin{array}{c}
			\lambda_1(\xbar) \\
			\lambda_2(\xbar)
		\end{array} \right] = \left[  \begin{array}{c}
		0\\
		\frac{\omega({p}_2(\xbar))}{G_{22}(\xbar)}
	\end{array} \right],
\end{align}
\noindent \textbf{Else if:} $G_{12}(\xbar) \omega ({p}_1(\xbar))-G_{11}(\xbar){p}_2(\xbar)<0,$
\begin{align}
	\label{eqn:ElseIfCondition}
	\left[  \begin{array}{c}
		\lambda_1(\xbar) \\
		\lambda_2(\xbar)
	\end{array} \right] = \left[  \begin{array}{c}
	\frac{\omega({p}_1(\xbar))}{G_{11}(\xbar)} \\
	0
\end{array} \right],
\end{align}

\noindent \textbf{Otherwise:}
\begin{align}
	\label{eqn:OtherwiseCondition}
	\left[  \begin{array}{c}
		\lambda_1(\xbar) \medskip \\
		\lambda_2(\xbar)
	\end{array} \right] = \left[  \begin{array}{c}
	\frac{\omega(G_{22}(\xbar){p}_1)(\xbar)-G_{21}(\xbar){p}_2(\xbar))}
	{G_{11}(\xbar)G_{22}(\xbar)-G_{12}(\xbar)G_{21}(\xbar)} \medskip \\
	\frac{\omega(G_{11}(\xbar){p}_2(\xbar)-G_{12}(\xbar){p}_1(\xbar))}
	{G_{11}(\xbar)G_{22}(\xbar)-G_{12}(\xbar)G_{21}(\xbar)}
\end{array} \right].
\end{align}

Because the Gram matrix is positive definite, for all $\xbar\in\mathcal{D}$, $G_{11}(\xbar)G_{22}(\xbar)-G_{12}(\xbar)G_{21}(\xbar)>0$. Using standard properties for the composition and product of locally Lipschitz continuous functions, each of the expressions in \eqref{eqn:IfCondition} -\eqref{eqn:OtherwiseCondition} is locally Lipschitz continuous on $\mathcal{D}$. Hence, the functions $\lambda_1(\xbar)$ and $\lambda_2(\xbar)$ are locally Lipschitz on each domain of definition and have well defined limits on the boundaries of their domains of definition relative to $\mathcal{D}$. If these limits agree at any point $\xbar$ that is common to more than one boundary, then $\lambda_1(\xbar)$ and $\lambda_2(\xbar)$ are locally Lipshitz continuous on $\mathcal{D}$. However, the limits are solutions to  \eqref{eqn:LuenbrgerFormulation}, and solutions to \eqref{eqn:LuenbrgerFormulation} are unique  \cite{LuenOptiBook}. Hence the limits agree at common points of their boundary\footnote{As an example, the only non-zero solutions of \eqref{eqn:LuenbrgerFormulation} occur when $p_2(\xbar)<0$, in which case,
	$G_{21}(\xbar)p_2(\xbar)-G_{22}(\xbar)p_1(\xbar)=0,$
	and  therefore \eqref{eqn:OtherwiseCondition} reduces to \eqref{eqn:IfCondition}. The other cases are similar.} (relative to $\mathcal{D}$) and the proof is complete.	
\end{proof}

\begin{remark}\label{rem:objective}
	If the objective function of $\mathcal{P}_2(x)$ is changed to $\frac{1}{2}\ubar^\top H\ubar+F^\top \ubar$ with $H$ an $(m+1)\times (m+1)$ positive definite matrix and $F$ an $(m+1)\times 1$ a column vector, then the modified QP is also locally Lipschitz continuous with respect to $x\in\D$.
	%\textbf{[Add remark on need of weight on the relaxation parameter.]}
\end{remark}

\section{Example}\label{sec:example}

In this section, the theoretical results of the paper are illustrated on adaptive cruise control (ACC). The \emph{lead} and \emph{following} vehicles are modeled as point-masses moving on a straight road with uncertain slope or grade  (\cite{ioannou1993autonomous}, \cite{astrom2010feedback}). The following vehicle is equipped with ACC, while the lead vehicle and the road act as disturbances to the following vehicle's performance objective of cruising at a given constant speed. The safety constraint is to maintain a safe following distance as specified by a time headway.

Let $v_l$ and $v_f$ be the velocity (in $m/s$) of the lead car and the following car, respectively, and $D$ be the distance (in $m$)  between the two vehicles. Let $x=(v_l,v_f,D)$  be the state of the system, whose dynamics can be described as
\begin{align}
\label{eqn:fgdynamics}
\left[\begin{array}{c}
\dot{v}_l\\
\dot{v}_f \smallskip\\
\dot{D}
\end{array}\right] &=\underbrace{
	\left[ \begin{array}{c}
	a_l \\
	-F_r/m\\
	v_l-v_f\end{array}\right]}_{f(x)}+\underbrace{
	\left[ \begin{array}{c}
	0\\
	g\Delta\theta\\
	0\end{array}\right]}_{\Delta f(x)}
+\underbrace{\left[ \begin{array}{c} 0 \\ 1/m\\0\end{array}\right]}_{\hat g(x)}u,
\end{align}
where $u$ and $m$ are the control input (in Newtons) and the mass (in $kg$) of the following car, respectively, $g$ is the gravitational constant (in $m^2/s$), $a_l$ is the acceleration (in $m^2/s$) of the lead car, $\Delta\theta$ is a perturbation to  $\dot{v}_f$ (reflecting unmodeled road grade or aerodynamic force), and $F_r=f_0 + f_1 v_f + f_2 v_f^2$ is the aerodynamic drag term (in Newtons) with constants $f_0$, $f_1$ and $f_2$ determined empirically. The values of $m$, $f_0$, $f_1$, and $f_2$ are the same as those in \cite{aaroncbfcdc14}.

Two constraints are imposed on the following car. The \emph{hard} constraint requires the following car to keep a safe distance from the lead car, which can be expressed as $ D/v_f \geq \tau_{des}$ with $\tau_{des}$ the desired time headway. Define the function $h=D-\tau_{des}v_f$, by which the hard constraint can be expressed as $h\geq 0$ and the set $\C$ can be defined by \eqref{eqn:superlevelsetC}-\eqref{eqn:superlevelsetC3}.  The \emph{soft} constraint requires that when adequate headway is assured, the following car achieves a desired speed $v_d$, which can be expressed as $v_f- v_d \to 0$, leading to the candidate CLF,  $V=(v_f-v_d)^2$.

%Note that
%\begin{align*}
%\dot{h}(x)&=\frac{1.8F_r}{m}+v_l-v_f-\frac{1.8u}{m}.
%\end{align*}

The controller is designed on the basis of the nominal model  $\dot{x}=f(x)+\hat g(x)u$ corresponding to  $\Delta f(x)=0$.  The hard constraint is encoded by the ZCBF condition \eqref{ineq:ZCBF} and the soft  constraint by the CLF condition \eqref{ineq:ESCLF}. The headway is selected as $\tau_{des}=1.8$ following
the ``half the speedometer rule'' (\cite{VogelHeadway03}).
The feedback controller $u(x)$ can then be obtained by the following QP
\begin{align*}
\ubar^*(x) &= \underset{\ubar = \left[ u, \delta\right]^\top \in  \R^2}{\operatorname{argmin}} \;  \frac{1}{2}
\ubar^T H \ubar + F^T \ubar \\
&\quad\mathrm{s.t.}\; A_{\mathrm{clf}} \ubar \leq b_{\mathrm{clf}}, \\
& \quad \quad\;\; A_{\mathrm{zcbf}}  \ubar \leq b_{\mathrm{zcbf}},
\end{align*}
where
\begin{align*}
H=2 \left[ \begin{array}{cc} 1/m^2 & 0 \\ 0 & p_{sc}  \end{array} \right] , \;
F=  -2 \left[ \begin{array}{c} F_r/m^2  \\ 0    \end{array} \right],
\end{align*}
as given in \cite{aaroncbfcdc14} with $p_{sc}$ the weight for $\delta$,
\begin{align*}
A_{\mathrm{clf}} &= \left[ \frac{2(v_f - v_d)}{m}, \; -1\right],\\
b_{\mathrm{clf}}&= \frac{2(v_f-v_d)}{m}F_r-(v_f -v_d)^2,
\end{align*}
%\begin{align*}
%A_{\mathrm{clf}} = \left[\psi_{1}(x), \; -1\right],&\;
%b_{\mathrm{clf}}= - \psi_{0}(x),
%\end{align*}
%with
%\begin{align*}
%\psi_{0}(x) &=-\frac{2(v_f-v_d)}{m}F_r+(v_f -v_d)^2,\\
%\psi_{1}(x) &= \frac{2(v_f - v_d)}{m},
%\end{align*}
and
\begin{align*}
A_{\mathrm{zcbf}}&=\left[-\frac{1.8}{m},\;  0\right],\\
b_{\mathrm{zcbf}}&=- \frac{1.8F_r}{m}-(v_l-v_f)+\alpha(h(x)).
\end{align*}

%By the results of Section \ref{sec:barrier} and Section \ref{sec:LipCon}, a control law designed for the unperturbed control system \eqref{eqn:controlsys} can be applied to the system \eqref{eqn:fgdynamics} guaranteeing the robust performance of the resulting closed-loop system.
%Specifically, for the unperturbed system \eqref{eqn:controlsys}, if a ES-CLF $V$ and a ZCBF $h$ with relative degree one are given,
%a locally Lipschitz continuous control law $u(x)$ can be obtained by solving QP $\mathcal{P}_2(x)$ (possibly with a more general objective function, see Remark \ref{rem:objective}). Then the following closed-loop system admits a well-defined solution
%\begin{align}\label{eqn:closedloop}
%\dot{x}&=f(x)+g(x)u(x).
%\end{align}

According to Proposition \ref{Prop:SetStability}, $V_\C$ defined in \eqref{Eq:V} equals to {$1.8v_f-D$} for points outside $\C$ and equals to $0$ for points inside $\C$. In the absence of perturbations, the input $u$ arising from solutions of the QP  ensures
$$
L_{f+\hat gu}V_\C\leq -\kappa V_\C,
$$
where the corresponding extended class $\Kinfinity$ function $\alpha$ is simply chosen as $\alpha(h)=\kappa h$ for some constant $\kappa>0$. For the perturbed system \eqref{eqn:fgdynamics}, the same input $u$ ensures
$$
L_{f+\hat gu+\Delta f}V_\C\leq -\kappa V_\C+L_{\Delta f}V_\C=-\kappa V_\C+1.8g\Delta\theta.
$$

By choosing the class $\mathcal{K}$ function $\gamma$ in Proposition \ref{Prop:Robust} as $\gamma(z)=\frac{1.8g}{\kappa}z$, the set $\mathcal{C}_{\gamma\left(\Vert \Delta\theta\Vert_\infty\right)}$ is asymptotically stable. Indeed, if $x\notin\mathcal{C}_{\gamma\left(\Vert \Delta \theta\Vert_\infty\right)}$, then $h(x)=D-1.8v_f<-\frac{1.8g}{\kappa}\Vert \Delta\theta\Vert_\infty$ and therefore,
\begin{align*}
L_{f+\hat gu+\Delta f}V_\C&\leq -\kappa V_\C+1.8g \Vert \Delta\theta\Vert_\infty\\
&=\kappa (D-1.8v_f)+1.8g \Vert \Delta\theta\Vert_\infty\\
&<-\kappa \frac{1.8g}{\kappa}\Vert \Delta\theta\Vert_\infty+1.8g\Vert \Delta\theta\Vert_\infty\\
&=0.
\end{align*}
Thus, for any $x\in\R^3\backslash \mathcal{C}_{\gamma\left(\Vert \Delta\theta\Vert_\infty\right)}$, $L_{f+gu+\Delta f}V_\C(x)<0$, which implies that the set $\mathcal{C}_{\gamma\left(\Vert \Delta\theta\Vert_\infty\right)}$ is asymptotically stable.

\begin{figure}[!thb]
	\centering
	\includegraphics[width=6cm]{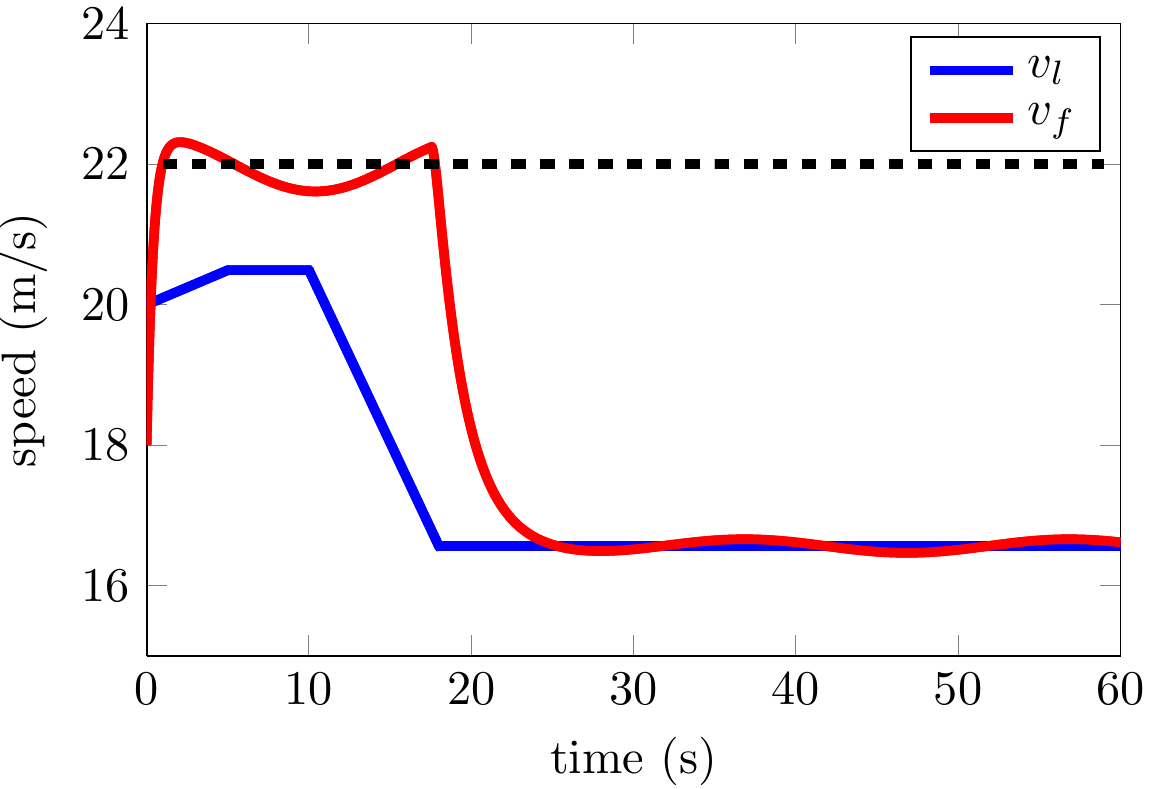}\\
	\includegraphics[width=6cm]{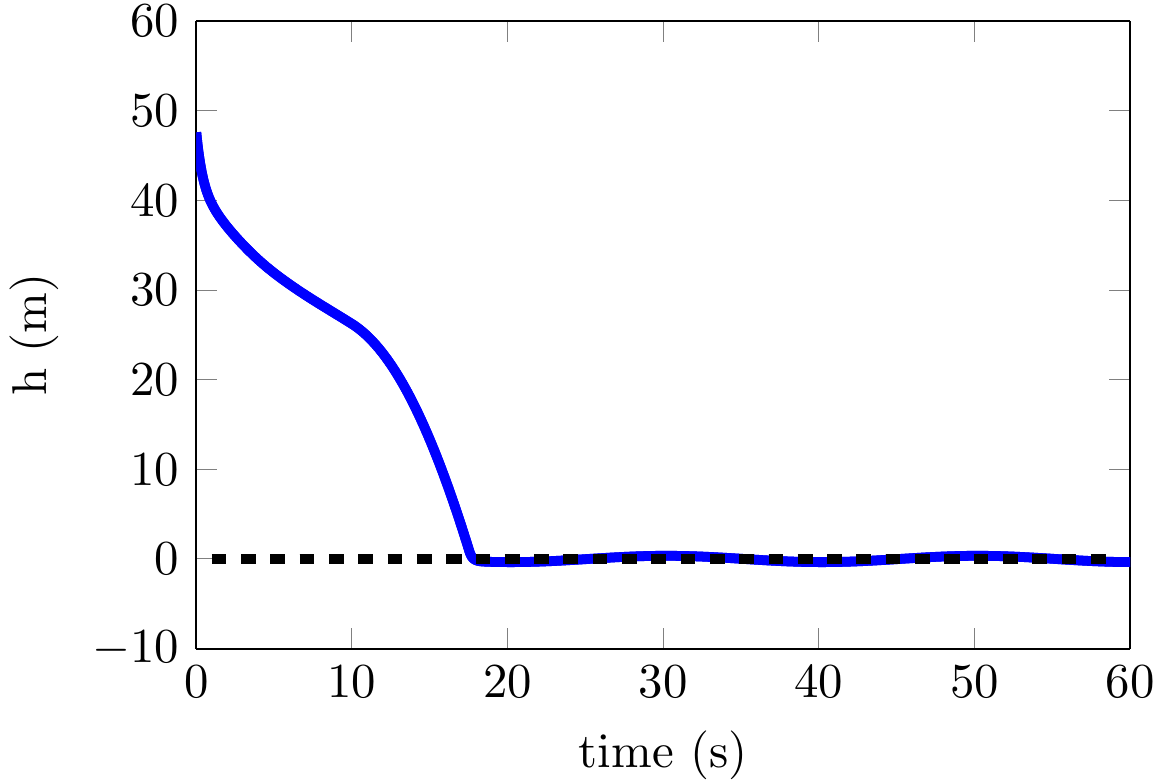}\\
	\includegraphics[width=6cm]{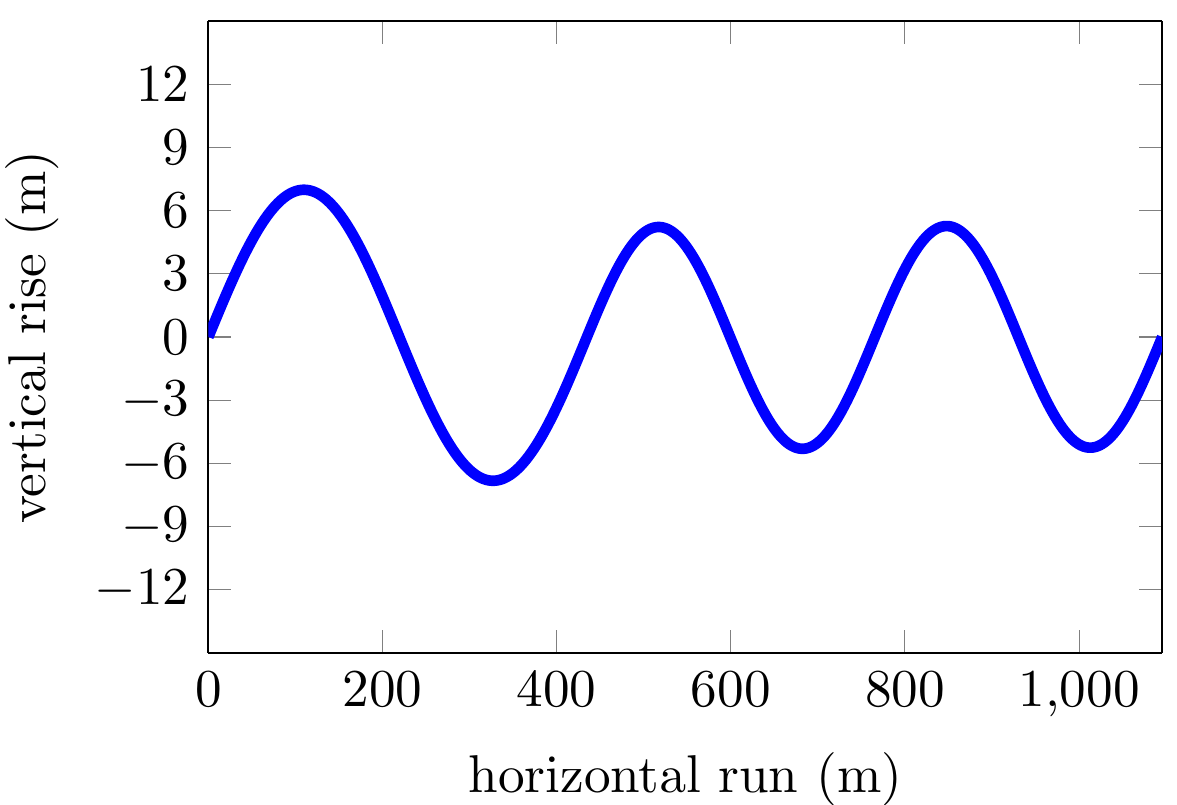}\\
	\caption{Simulation results when choosing $\kappa=5$, $\|\Delta \theta\|_{\infty}=0.1$ and initial states $v_l(0)=20$, $v_f(0)=18$, $D(0)=80$. (top) speed of the two cars; (middle) evolution of $h=D-1.8v_f$; (bottom) the vertical rise of the road with respect to the horizontal run of the car.}\label{fig:casestudy}
\end{figure}

\begin{figure}[!thb]
	\centering
	\includegraphics[width=6cm]{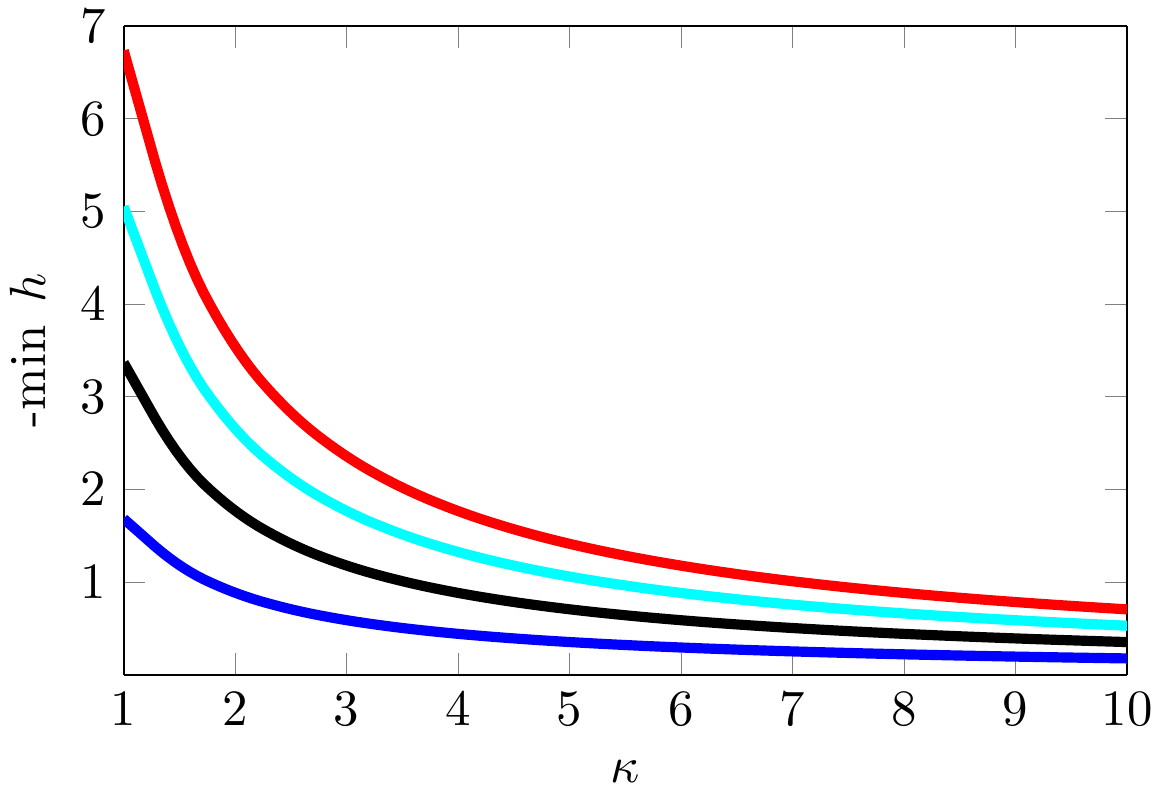}\\
	\includegraphics[width=6cm]{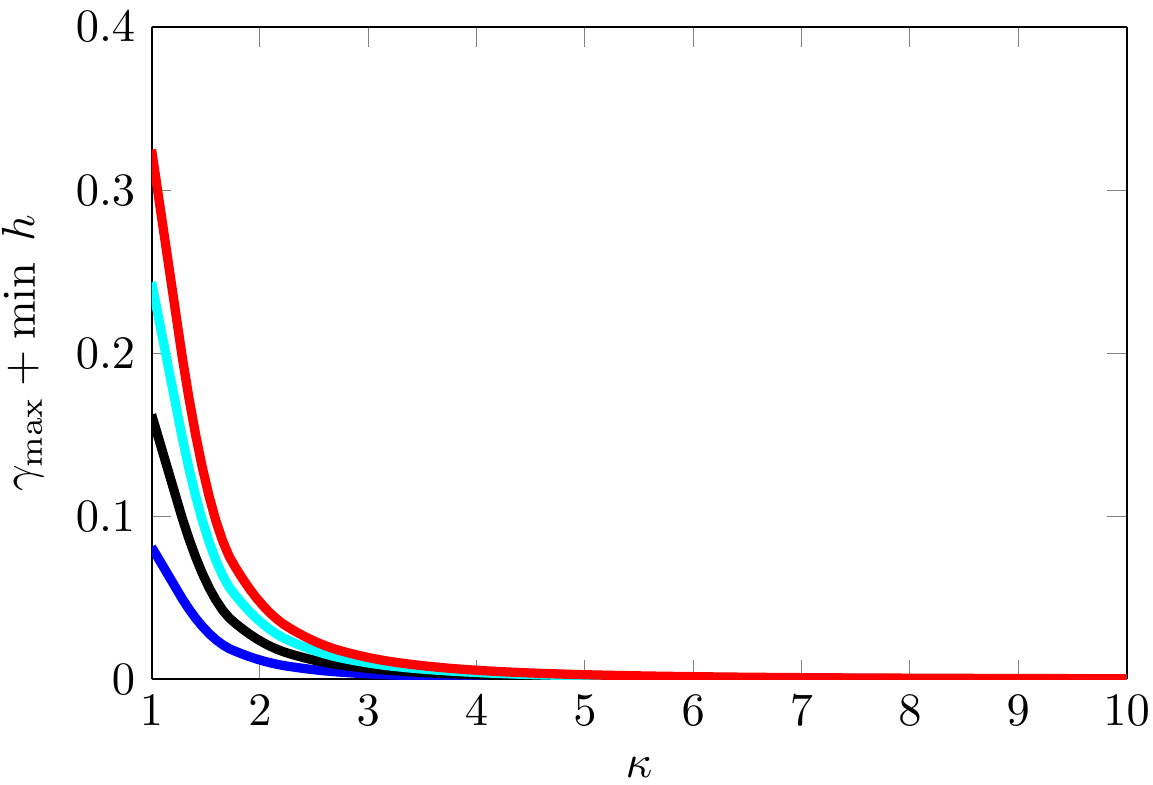}\\
	\includegraphics[width=6cm]{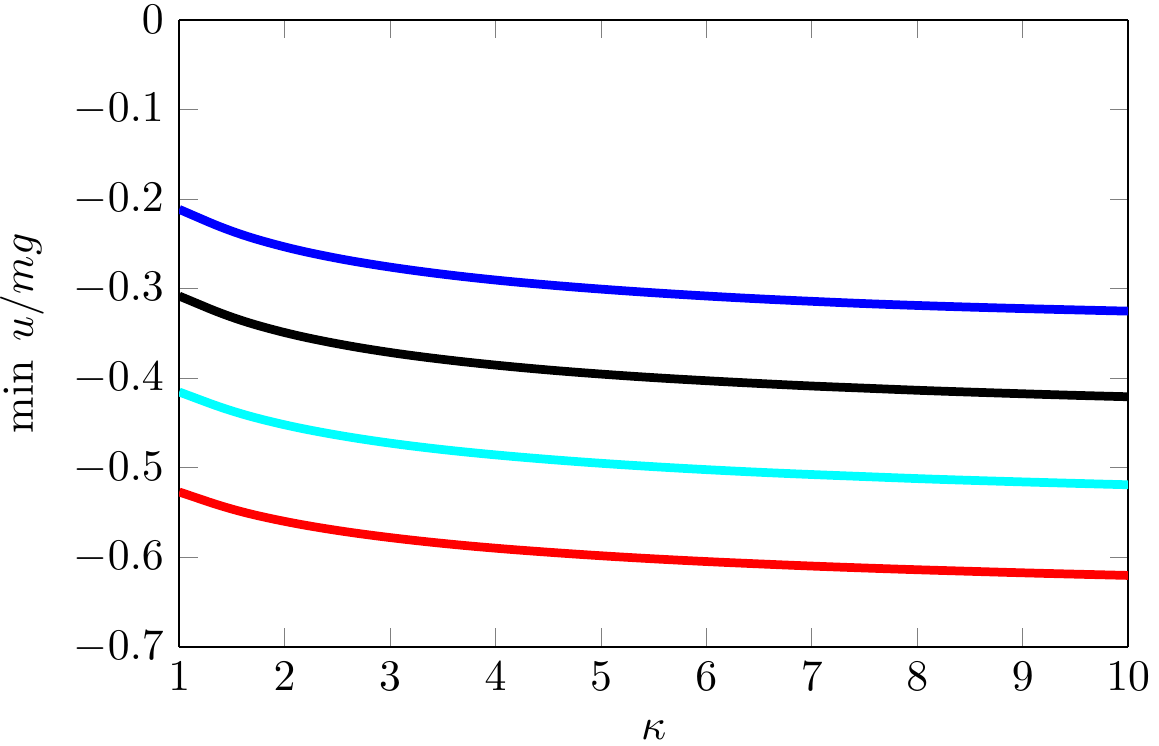}\\
	\includegraphics[width=3cm]{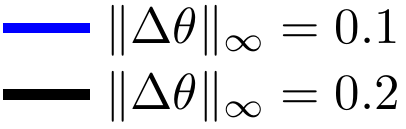}
	\includegraphics[width=3cm]{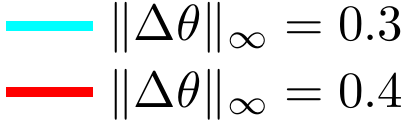}\\
	%\caption{(top) comparison of $-\inf_{t\in[0,60]} h(x(t))$ and
	%$\ep_{\min }$; (bottom) the maximal braking force $\inf_{t\in[0,60]}u(t)$ proportional to $mg$.}
	\caption{Tradeoff analysis in terms of road grade uncertainty and speed of convergence to the safe set. (top) $-\min h$ increases as $\kappa$ decreases and $\Vert \Delta\theta\Vert_\infty$ increases; (middle) positiveness of the discrepancies of $\gamma_{\max}$ and $\min h$ implies $x$ is within the set $\C_{\gamma_{\max}}$; (bottom) the magnitude of the braking force $u$ increases as $\Vert \Delta\theta\Vert_\infty$ and $\kappa$ increases. According to the Guinness Book of World Records, the steepest street in the world is Baldwin Street in New Zealand, with a grade of 38\%.}
	\label{fig:dischange}
	%\caption{(top) Comparison of $-\inf_{t\in I(x_0)} h(x(t))$ and
	%$\ep_{\min }$; (bottom) $\inf_{t\in I(x_0)}u(t)$ proportional to $mg$.}\label{fig:dischange}
\end{figure}

%\begin{figure}[!thb]
%	\centering
%	\includegraphics[width=.73\textwidth]{fig_hardconstraint_3d.pdf}\\
%	\includegraphics[width=.75\textwidth]{fig_epmin_3d.pdf}\\
%	\includegraphics[width=.77\textwidth]{fig_force_3d.pdf}\\
%	\caption{{\color{blue}3D demonstration}}
%	
%	%\caption{(top) comparison of $-\inf_{t\in[0,60]} h(x(t))$ and
%	%$\ep_{\min }$; (bottom) the maximal braking force $\inf_{t\in[0,60]}u(t)$ proportional to $mg$.}
%	%\caption{(top) Comparison of $-\inf_{t\in I(x_0)} h(x(t))$ and
%	%$\ep_{\min }$; (bottom) $\inf_{t\in I(x_0)}u(t)$ proportional to $mg$.}\label{fig:dischange}
%\end{figure}

Figure \ref{fig:casestudy} shows the time evolution of $v_l(t)$ and $v_f(t)$, the evolution of the specification $h(x(t))$, the change of the road slope when $\kappa=5$, the
perturbation $\Delta \theta(t)=0.1\cos(2\pi t/20)$, and
the desired speed $v_d=22$. The initial state is $v_l(0)=20$, $v_f(0)=18$, and $D(0)=80$. To simplify the discussion we denote $\gamma\left(\Vert \Delta\theta\Vert_\infty\right)$ by $\gamma_{\max}$ which is $\gamma_{\max}=0.3532$ for $\Vert \Delta\theta\Vert_\infty=0.1$, i.e., the maximum headway distance error is $0.3532$m. The top plot of Fig.~\ref{fig:casestudy} shows that the following car first accelerates to approximately its desired speed $v_d$. The vehicle then decelerates to and maintains the same final speed as the lead car in order to maintain a safe headway. Note that due to the unmeasured perturbation in road grade, the achieved tracking speeds are in a neighborhood of $v_d$ and the lead car's final speed. The middle plot shows that the values of $h$ are greater than $-0.3525$, which implies that $x$ is within the set $\C_{\gamma_{\max}}$. The bottom plot shows the vertical rise of the road with respect to the horizontal run of the car, assuming the perturbation term is exclusively interpreted as the change of road slope.

Figure \ref{fig:dischange} shows the quantities $-\min h$, the amount the safety condition is violated in meters, $\gamma_{\max}+\min h$, the tightness of the error bound in meters, and $\min u/mg$, the braking effort in fractions of $g$,
as $\kappa$ ranges from $1$ to $10$ (rate of convergence back to safe set) and $\Vert \Delta\theta\Vert_\infty$ ranges from 10\% to 40\% (road grade perturbation), where
%$$ \min h:=\min_{0\le t \le 60}~~h(x(t)), ~~~\min u:=\min_{0\le t \le 60}~~u(t)\\$$
% \begin{align*}
%\min h&:=\min \{ h(x(t))~|~ 0\le t \le 60\} \\
%\min u &:= \min \{ u(t)~|~ 0\le t \le 60\}.
%\end{align*}
\begin{align*}
\min h&:=\min_{0\le t \le 60}~h(x(t)) \\
\min u&:=\min_{0\le t \le 60}~u(t).
\end{align*} 
Note that larger $\kappa$ means a stricter barrier function condition, while larger $\Vert \Delta\theta\Vert_\infty$ means more uncertainty in the dynamics. The evolution of the road grade perturbation is given by $\Delta \theta(t)=0.1K\cos(2\pi t/20)$ for a constant $K>0$, which implies $\|\Delta \theta\|_{\infty}=0.1K$.
The top plot shows that $-\min h$ increases as $\kappa$ decreases or $\Vert \Delta\theta\Vert_\infty$ increases, which is intuitive because with a weaker barrier function condition or larger perturbations, the specification $h>0$ is more likely to be violated.
The middle plot shows that the discrepancies between $\gamma_{\max}$ and $\min h$ are positive, which implies that $x$ is always within the set $\C_{\gamma_{\max}}$ as Proposition \ref{Prop:SetStability} guarantees. The bottom plot shows that the magnitude of the braking force $u$ increases as $\Vert \Delta\theta\Vert_\infty$ or $\kappa$ increases.

%Figure \ref{fig:dischange} shows how $-\min h:=-\min_{1\le t \le 60}h(x(t))$, $\gamma_{\max}+\min h$ and $\min u/mg$ vary as
%$\kappa$ ranges from $1$ to $10$ (rate of convergence back to safe set) for different values of $\Vert \Delta\theta\Vert_\infty$ (road grade perturbation). Note that larger $\kappa$ means a stricter barrier function condition, while larger $\Vert \Delta\theta\Vert_\infty$ means more uncertainty in the dynamics. The evolution of the perturbation is given by $\Delta \theta(t)=0.1K\cos(2\pi t/20)$ for a constant $K>0$, which implies $\|\Delta \theta\|_{\infty}=K$ \textcolor{red}{(Should this be $0.1K$? Paulo 04/29/2015)}.
%The top plot shows that $-\min h$ increases as $\kappa$ decreases and $\Vert \Delta\theta\Vert_\infty$ increases, which is intuitive because with a stricter barrier function condition and larger perturbations, the specification $h>0$ is more likely to be violated, which results in larger  $-\min h$.
%The middle plot shows that the discrepancies between $\gamma_{\max}$ and $\min h$ are positive, which implies that $x$ is always within the set $\C_{\gamma_{\max}}$ as Proposition \ref{Prop:SetStability} guarantees. The bottom plot shows that the magnitude of the braking force $u$ increases as $\Vert \Delta\theta\Vert_\infty$ and $\kappa$ increases.

%{\color{blue} Jessy found that the steepest road in the world is about 37 percent gradient, less than $0.4$.}

%\section{Experimental Results}\label{sec:experiment}

\section{Conclusions}\label{sec:conclusions}
This paper defined (control) zeroing barrier functions for a given set and investigated their robustness properties under model perturbations. In particular, when the barrier function was designed to be negative on the complement of the closure of a safe set, and its derivative along solutions of the model was positive, a Lypaunov analysis showed that the set was automatically locally asymptotically stable. This led to various Input-to-State Stability (ISS) results in the presence of model perturbations. For this result to hold, it was important to consider barrier functions that vanish on the set boundary (i.e., zeroing barrier functions) rather than barrier functions that tend to infinity on the set boundary (i.e. reciprocal barrier functions). The reason is that ``there are two sides of zero'' and only ``one side of infinity.'' More formally speaking, if a perturbation (or model error) makes it impossible to satisfy the invariance condition for a reciprocal barrier function, then the solution of the model must cease to exist because the control input must become unbounded as well; see Sect.~III.B of \cite{aaroncbfcdc14}, eqn. (CBF). On the other hand, if a perturbation (or model error) makes it impossible to satisfy the invariance condition for a zeroing barrier function, then the solution can cross the set boundary without the control input becoming unbounded.

A second result presented conditions that guarantee local Lipschitz continuity of the solution of a Quadratic Program (QP) that  mediates safety (represented as a control barrier function (CBF)) and a control objective (represented as a control Lyapunov function (CLF)). A uniform relative degree condition on the CBF and relaxation of the inequality required for a CLF were shown to provide local Lipschitz continuity of the resulting feedback control law, and hence local existence and uniqueness of solutions of the associated closed-loop system. This result is applicable to both types of barrier functions.

Future studies will consider control zeroing barrier functions with constraints on the inputs, as in \cite{aaroncbfcdc14}. There are many interesting open questions on existence, computation, and composition, as well as applications to systems of greater complexity than ACC.

\bibliographystyle{IEEEtran}
\bibliography{arxiv_adhs_ref}
%\bibliography{./composite_references,./new_ref,./barrier,./SOS,./smallgain,./Rfunction,./vehicle,
%	./invariance_control,./literature_review,./reference_governor}
\end{document}